\documentclass[12pt]{article}

\usepackage{latexsym,amsfonts,amsmath,amsthm,amssymb, epsfig}
\usepackage{color}

\usepackage[vcentering,dvips]{geometry}
\newtheorem{thm}{Theorem}

\newtheorem{prop}[thm]{Proposition}
\newtheorem{cor}{Corollary}

\newtheorem{remark}{Remark}

\theoremstyle{remark}

\theoremstyle{definition}
\newtheorem{dfn}[thm]{Definition}

\newcommand{\N}{\mathbb{N}}

\newcommand{\D}{\mathbb{D}}

\newcommand{\sgn}{{\rm sign}}

\newcommand{\ex}{{\rm ext}}
\newcommand{\bx}{{\rm box}}
\newcommand{\mx}{{\rm mix}}
\newcommand{\qmc}{{\rm QMC}}

\newcommand{\rd}{\,{\rm d}}
\newcommand{\bsx}{{\boldsymbol x}}
\newcommand{\bsy}{{\boldsymbol y}}
\newcommand{\bst}{{\boldsymbol t}}
\newcommand{\bsa}{{\boldsymbol a}}
\newcommand{\bsb}{{\boldsymbol b}}

\newcommand{\uu}{\mathfrak{u}}

\newcommand{\cP}{\mathcal{P}}
\newcommand{\cA}{\mathcal{A}}

\title{Extreme $L_p$ discrepancy, numerical integration and the curse of dimensionality}  

\author{Erich Novak and  Friedrich Pillichshammer}

\date{\today} 

\begin{document}

\maketitle

\begin{abstract}
The classical notion of extreme $L_p$ discrepancy is a quantitative measure for the irregularity 
of distribution of finite point sets in the $d$-dimensinal unit cube. In this paper we find a dual integration problem whose worst-case error is exactly the extreme $L_p$ discrepancy of the underlying integration nodes. Studying this integration problem we show that the extreme $L_p$ discrepancy suffers from the curse of 
dimensionality for all $p \in (1,\infty)$. It is known that the problem is tractable for $p=\infty$; the case $p=1$ stays open.
\end{abstract}

\centerline{\begin{minipage}[hc]{130mm}{
}
\end{minipage}}

%E 
\section{Introduction}\label{sec:intro}

For a set $\cP$ consisting of $N$ points $\bsx_1,\bsx_2,\ldots,\bsx_N$ in the $d$-dimensional unit-cube $[0,1)^d$ the local discrepancy with respect to a measurable test set $C$ is defined as $$\Delta_{\cP}(C)=\frac{|\{k \in \{1,2,\ldots,N\}\ : \ \bsx_k \in C\}|}{N}-{\rm volume}(C).$$ Using a class of test sets that are suitably parametrized, the local discrepancy can be considered as a function of this parameter. This is usually emphasized by the nomenclature ``local discrepancy function''. Taking a norm of the local discrepancy function leads to the notion of a discrepancy, which is a quantitative measure for the irregularity of distribution of the point set $\cP$ with respect to the uniform measure and the class of test sets. Many notions of discrepancy are based on a suitable $L_p$-norm and often a discrepancy is related to the worst-case error of a quasi-Monte Carlo (QMC) algorithm for numerical integration of a function from suitable function spaces.
The most prominent example is the usual star $L_p$ discrepancy, which uses as test sets the class 
of anchored hypercubes in $[0,1]^d$ of the form $[\boldsymbol{0},\bsb)=[0,b_1) \times \ldots 
\times [0,b_d)$ with $\bsb=(b_1,\ldots,b_d) \in [0,1]^d$. 
In this case the relation between star $L_p$ discrepancy and the 
corresponding integration problem such that the worst-case error equals the star $L_p$ discrepancy 
is well understood (see \cite[Chapter~9]{NW10} or the recent papers \cite{NP23,NP25}). 
We call this behavior a ``discrepancy-integration duality''. 

The inverse of a discrepancy for dimension $d$ and error threshold $\varepsilon \in (0,1)$ is the minimal number of points in $[0,1)^d$ required such that the minimal discrepancy is less or equal an $\varepsilon$ share of the initial discrepancy. For applications of QMC to high-dimensional problems, it is of great importance whether the inverse grows exponentially with the dimension (the curse of dimensionality) or not (tractability). Again, for the classical star $L_p$ discrepancy the behavior is for $p \in (1,\infty]$ well understood. We have tractability for $p=\infty$ (see \cite{hnww}) and the curse of dimensionality for $p \in (1,\infty)$ (see \cite{NP25}). The case $p=1$ is still open. 

In this paper we study the extreme $L_p$ discrepancy which uses as class of test sets all hypercubes $[\bsa,\bsb) \subseteq [0,1)^d$. At least for $p=\infty$, this discrepancy concept is a classic one that is discussed in the literature (see e.g. \cite{DT97,KN74}), but the more general $L_p$ case has also been investigated (see e.g. \cite{Gne05,HW12,MC94} or \cite[Sec.~9.5.5]{NW10}). However, in comparison to the star discrepancy, relatively little is known for extreme discrepancy. We know neither of a corresponding integration problem, such that a discrepancy-integration duality holds true, nor do we know about the behavior of the inverse extreme $L_p$ discrepancy with respect to the dimension for $p \not\in \{2,\infty\}$.

In this article, we investigate these questions and identify an integration problem where the worst-case error and the extreme $L_p$ discrepancy coincide (discrepancy-integration duality). Furthermore, we prove that the extreme $L_p$ discrepancy for $p \in (1,\infty)$ suffers from the curse of dimensionality.

The paper is organized as follows. In Section~\ref{sec:disc} we give a rigorous definition of extreme $L_p$ discrepancy and its extension to generalized extreme $L_p$ discrepancy.  A related dual integration problem is introduced in Section~\ref{sec:int}. The main result of this section, which is Theorem~\ref{thm:main}, states that the worst-case error for integration with arbitrary linear rules (which comprise the concept of QMC algorithms) equals the generalized extreme $L_p$ discrepancy. Furthermore, we study in detail the initial error of our integration problem and figure out a so-called worst-case function, i.e., a function whose integral equals the initial error. In Section~\ref{sec:curse} we prove that the extreme $L_p$ discrepancy for $p \in (1,\infty)$ as well as the dual integration problem suffers from the curse of dimensionality. The proof of this result is given in Section~\ref{sec:proof}.

\section{Extreme $L_p$ discrepancy} \label{sec:disc}

The {\it extreme $L_p$ discrepancy} uses as test sets arbitrary hypercubes contained in the unit cube. Let $$\D_d:=\{(\bsa,\bsb) \in [0,1]^d \times [0,1]^d \ : \ \bsa \le \bsb\},$$ where for $\bsa=(a_1,a_2,\ldots,a_d)$ and $\bsb=(b_1,b_2,\ldots,b_d)$ in $[0,1]^d$ we write $\bsa \leq \bsb$ if and only if $a_j \le b_j$ for all $j \in \{1,\ldots,d\}$. For $(\bsa,\bsb) \in \D_d$ let $[\bsa,\bsb)=[a_1,b_1)\times [a_2,b_2) \times \ldots \times [a_d,b_d)$.  Throughout, let ${\bf 1}_{[\bsa,\bsb)}$ denote the indicator function of $[\bsa,\bsb)$. Then
$$\Delta_{\cP}([\bsa,\bsb))=\frac{1}{N} \sum_{k=1}^N \mathbf{1}_{[\bsa,\bsb)}(\bsx_k) - \prod_{j=1}^d (b_j-a_j)$$ and
the extreme $L_p$ discrepancy of $\cP$ is then defined as
$$L_{p,N}^{\ex}(\cP):=\left(\int_{\D_d} |\Delta_{\cP}([\bsa,\bsb))|^p\rd (\bsa,\bsb)\right)^{1/p},$$ for $p \in [1,\infty)$ and
$$L_{\infty,N}^{\ex}(\cP):=\sup_{(\bsa,\bsb) \in \D_d}  |\Delta_{\cP}([\bsa,\bsb))|,$$ for $p =\infty$.

An extension of the classical definition additionally uses a set of real weights $\cA=\{c_1,\ldots,c_N\}$ corresponding 
to the points $\cP=\{\bsx_1,\ldots,\bsx_N\}$. Then the generalized discrepancy function with weights $\cA$ is
\[
\Delta_{\cP,\cA}([\bsa,\bsb))
   = \sum_{k=1}^N c_k\,\mathbf{1}_{[\bsa,\bsb)}(\bsx_k) 
     - \prod_{j=1}^d (b_j-a_j),
\]
and the generalized extreme $L_p$ discrepancy with weights $\cA$ is 
\[
L_{p,N}^{\ex}(\cP,\cA)
  = \left(\int_{\D_d}
        |\Delta_{\cP,\cA}([\bsa,\bsb))|^p \rd(\bsa,\bsb) \right)^{1/p},
\]
for $p\in[1,\infty)$ with the usual modification for $p=\infty$. 
If $c_1=\dots=c_N=1/N$, we recover the classical extreme $L_p$ discrepancy $L_{p,N}^{\ex}(\cP)$ of the point set $\cP$.

It is easily seen that the so-called initial extreme $L_p$ discrepancy, i.e., the extreme $L_p$ discrepancy of the emptyset, or, in other words, the $L_p$-norm of the volume function $(\bsa,\bsb) \in \D_d \mapsto {\rm volume}([\bsa,\bsb))$, equals $$L_{p,0}^{\ex}=\left\{ 
\begin{array}{ll}
((p+1)(p+2))^{-d/p} & \mbox{if $p \in [1,\infty)$,}\\
1 & \mbox{if $p=\infty$.} 
\end{array}
\right.$$
This is clear for $p=\infty$. For $p \in [1,\infty)$ we have
\begin{align*}
(L_{p,0}^{\ex})^p = & \int_{\D_d} \prod_{j=1}^d (b_j-a_j)^p \rd (\bsa,\bsb) 
= \left(\int_0^1 \int_a^1 (b-a)^p \rd b \rd a \right)^d\\
= & \frac{1}{((p+1)(p+2))^d},
\end{align*}
and the result follows by taking the $p$-th root.

\section{A related dual integration problem}\label{sec:int}

Throughout we assume that $p,q \in [1,\infty]$ are H\"older conjugates, i.e., $
\frac1p+\frac1q=1$. For $p,q \in (1, \infty)$ this implies that $\frac{q}{p}=q-1$.

\paragraph{Representation-space definition via box indicators.}

For $d$-fold hypercubes  $[\bsa,\bsb]=[a_1,b_1]\times \ldots \times [a_d,b_d]$ we use the indicator function
\[
{\bf 1}_{[\bsa,\bsb]}(\bsx) = \prod_{j=1}^d {\bf 1}_{[a_j,b_j]}(x_j) \quad \mbox{for $\bsx=(x_1,\ldots,x_d)$.}
\]

\begin{dfn}[Representation-space definition of $F_{d,q}$]\label{def:box}
Define a linear operator $T_d$ on $L_q(\D_d)$ via 
\begin{equation}\label{eq:rep}
(T_d c)(\bsx):=\int_{\D_d} c(\bsa,\bsb)\, {\bf 1}_{[\bsa,\bsb]}(\bsx) \rd (\bsa,\bsb)
\quad\text{for }c\in L_q(\D_d).
\end{equation}
Define $$F_{d,q}:=\{f \ : \ c\in L_q(\D_d), \ f=T_dc\}$$ and equip this space with the norm 
\begin{equation}\label{eq:norm-rep}
\|f\|_{F_{d,q},\bx}:=\inf\left\{\|c\|_{L_q(\D_d)}:\;c\in L_q(\D_d), \ f=T_dc\right\}.
\end{equation}
\end{dfn}

\begin{remark}\rm
For every $c \in L_q(\D_d)$ the function $f=T_d c$ is well defined a.e., belongs to $L_q([0,1]^d)$, is absolutely continuous in each variable, and satisfies the boundary condition $f(\bsx)=0$ whenever some $x_j \in \{0,1\}$.   
\end{remark}

\begin{remark}\rm
We confess that we do not know whether the norm $\|\cdot\|_{F_{d,q},\bx}$ is a tensor-product norm. In any case, for functions $f(\bsx)=f_1(x_1)\cdots f_d(x_d)$ with univariate functions $f_1,\ldots,f_d$ we have $\|f\|_{F_{d,q},\bx} \le  \|f_1\|_{F_{1,q},\bx} \cdots \|f_d\|_{F_{1,q},\bx}$.
\end{remark}

\begin{remark}\label{remark3}\rm
Likewise, one might equip $F_{d,q}$ with a mixed Sobolev norm  
\[
\|f\|_{F_{d,q},\mx} := \| D f\|_{L_q([0,1]^d)},
\]
with the mixed derivative $D f:=\partial_1\cdots\partial_d f$ (which exists in the weak sense and belongs to $L_q([0,1]^d)$).
The two norms are equal for $q=2$, i.e. $\| f \|_{F_{d,q},\bx} = \|f\|_{F_{d,q},\mx}$, and it can be shown that the two norms are always equivalent, although with implied constants that depend 
on the dimension $d$. In more detail, 
\begin{equation}\label{norm:equiv}
2^{-d} \|f\|_{F_{d,q},\mx} \le \|f\|_{F_{d,q},\bx} \le 2^d  \|f\|_{F_{d,q},\mx}.
\end{equation}
See  Appendix~\ref{sec:appA} for details.
\end{remark}

In the following we will work with the space $F_{d,q}$ equipped with the box-norm $\|\cdot\|_{F_{d,q},\bx}$.

\paragraph{Integration problem and worst-case error.}

As above, let $\cP=\{\bsx_1,\ldots,\bsx_N\}$ be an $N$-element point set in $[0,1)^d$ and 
let $\cA=\{c_1,\ldots,c_N\}$ be corresponding weights. For $f \in F_{d,q}$ let
\begin{equation}\label{def:I:Q}
I(f):=\int_{[0,1]^d} f(\bsx)\rd \bsx
\qquad \mbox{and}\qquad
Q_{\cP,\cA}(f):=\sum_{k=1}^N c_k f(\bsx_k).
\end{equation}
If we restrict to weights $c_1=\ldots =c_N=1/N$, we denote these by $\cA^{\qmc}$, 
then $Q_{\cP,\cA^{\qmc}}(f):=\frac{1}{N}\sum_{k=1}^N f(\bsx_k)$, which is a so-called quasi-Monte Carlo (QMC) algorithm. 
Define the error functional
\[
L(f):=I(f)-Q_{\cP,\cA}(f).
\]
With $\mu=\sum_{k=1}^N c_k\delta_{\bsx_k}-\lambda$,  where $\delta_{\bsx_k}$ denotes the Dirac measure concentrated in $\bsx_k$ and $\lambda$ the Lebesgue measure on $[0,1]^d$,  we have
\begin{equation}\label{eq:L-mu}
L(f)=\int_{[0,1]^d} f(\bsx) \rd\left(\lambda-\sum_{k=1}^N c_k\delta_{\bsx_k}\right)(\bsx)
= -\int_{[0,1]^d} f(\bsx) \rd\mu(\bsx).
\end{equation}
The worst-case error in $F_{d,q}$ (with respect to the box-norm $\|\cdot\|_{F_{d,q},\bx}$) is
\[
e(Q_{\cP,\cA};F_{d,q}) := \sup_{\|f\|_{F_{d,q},\bx}\le 1}|L(f)|.
\]

\begin{thm}[Discrepancy-integration duality]\label{thm:main}
Let $p,q \in [1,\infty ]$ be H\"older conjugates. For every point set 
$\cP=\{\bsx_1,\ldots,\bsx_N\}$ in $[0,1)^d$ and arbitrary corresponding real weights $\cA=\{c_1,\ldots,c_N\}$ we have
\[
e(Q_{\cP,\cA};F_{d,q}) \;=\; L_{p,N}^{\ex}(\cP,\cA).
\]
\end{thm}

\begin{proof}
Let $f\in F_{d,q}$ and choose $c\in L_q(\D_d)$ representing $f$ as in \eqref{eq:rep}.
Using \eqref{eq:L-mu} and Fubini, we compute
\begin{align*}
L(f)
&= -\int_{[0,1]^d}  f(\bsx) \rd\mu(\bsx)
= -\int_{[0,1]^d} \left(\int_{\D_d} c(\bsa,\bsb) \,{\bf 1}_{[\bsa,\bsb]}(\bsx)\rd (\bsa,\bsb)\right) \rd \mu(\bsx)\\
&= -\int_{\D_d} c(\bsa,\bsb)\left(\int_{[0,1]^d} {\bf 1}_{[\bsa,\bsb]}(\bsx)\,d\mu(\bsx)\right) \rd (\bsa,\bsb)\\
&= -\int_{\D_d} c(\bsa,\bsb) \,\mu([\bsa,\bsb]) \rd (\bsa,\bsb)\\
&= -\int_{\D_d} c(\bsa,\bsb)\,\Delta_{\cP,\cA}([\bsa,\bsb]) \rd (\bsa,\bsb).
\end{align*}
Now, by Hölder's inequality,
\[
|L(f)|
\le \|c\|_{L_q(\D_d)}\,\|\Delta_{\cP,\cA}\|_{L_p(\D_d)}
= \|c\|_{L_q(\D_d)}\,L_{p,N}^{\ex}(\cP,\cA).
\]
Taking the infimum over all representations of $f$ yields
\[
|L(f)| \le \|f\|_{F_{d,q},\bx}\,L_{p,N}^{\ex}(\cP,\cA).
\]
Hence $$e(Q_{\cP,\cA};F_{d,q})\le L_{p,N}^{\ex}(\cP,\cA).$$

It remains to prove equality. First we assume $p,q \in (1, \infty)$ and discuss the needed modifications 
for the endpoints later. If $\Delta_{\cP,\cA} \equiv 0$ then $L_{p,N}^{\ex}(\cP,\cA)=0$ and also $L(f)\equiv 0$, 
so equality holds. Assume now $\Delta_{\cP,\cA}\not\equiv 0$. Define 
\begin{equation}\label{eq:cstar}
c^{\ast}(\bsa,\bsb):=\frac{|\Delta_{\cP,\cA}([\bsa,\bsb])|^{p-2} \Delta_{\cP,\cA}([\bsa,\bsb])}{\|\Delta_{\cP,\cA}\|_{L_p(\D_d)}^{p-1}}
\quad\text{so that}\quad
\|c^{\ast}\|_{L_q(\D_d)}=1.
\end{equation}
Let $f^\ast=T_d c^\ast$ with coefficient $c^\ast$. Then $\|f^\ast\|_{F_{d,q},\bx}\le 1$ by definition of the norm \eqref{eq:norm-rep}, and
\begin{align*}
L(f^\ast)
= & -\int_{\D_d} c^\ast(\bsa,\bsb)\,\Delta_{\cP,\cA}([\bsa,\bsb])\rd (\bsa,\bsb) \\
= & -\frac{1}{\|\Delta_{\cP,\cA}\|_{L_p(\D_d)}^{p-1}}\int_{\D_d}|\Delta_{\cP,\cA}(\bsa,\bsb)|^p \rd (\bsa,\bsb)\\
= & -\|\Delta_{\cP,\cA}\|_{L_p(\D_d)}.
\end{align*}
Therefore $|L(f^\ast)|=\|\Delta_{\cP,\cA}\|_{L_p(\D_d)}=L_{p,N}^{\ex}(\cP,\cA)$ and hence
$$e(Q_{\cP,\cA};F_{d,q})\ge L_{p,N}^{\ex}(\cP,\cA).$$ Combining with the upper bound yields the identity.

For the case $p=1$ we take $c^*(\bsa,\bsb) = \sgn (\Delta_{\cP,\cA} (\bsa,\bsb))$ and obtain 
again $|L(f^\ast)|=\|\Delta_{\cP,\cA}\|_{L_p(\D_d)}=L_{p,N}^{\ex}(\cP,\cA)$.
Assume now that  $p=\infty$ and  $q=1$. Let
\[
M:=\Vert \Delta_{\cP,\cA} \Vert_{L_{\infty}(\D_d)}.
\]
For any $\varepsilon>0$, define
\[
E_\varepsilon
:=\{(\bsa,\bsb): |\Delta_{\cP,\cA}(\bsa,\bsb)| \ge M-\varepsilon\}.
\]
Assume $m(E_\varepsilon)>0$.
Define
\[
c_\varepsilon (\bsa, \bsb) 
=\frac{ \sgn (\Delta_{\cP,\cA} (\bsa, \bsb) )\,  {\mathbf 1}_{E_\varepsilon} (\bsa, \bsb) }{m(E_\varepsilon)}.
\]
Then $\Vert c_\varepsilon\Vert_{L_1}=1$ and
\[
\int_{\D_d}  c_\varepsilon (\bsa,\bsb)\,\Delta_{\cP,\cA}([\bsa,\bsb])\rd (\bsa,\bsb )
\ge M-\varepsilon.
\]
Letting $\varepsilon\to0$ gives
\[
e(Q_{\cP,\cA};F_{d,1})
=\Vert  \Delta_{\cP,\cA}\Vert_{L_\infty}=L_{\infty,N}^{\ex}(\cP,\cA).
\]
\end{proof}

\paragraph{Minimal discrepancy and minimal worst-case error.}

For $d,N \in \N$ the $N$-th minimal extreme $L_p$ discrepancy in dimension $d$ is defined as $${\rm disc}_p^{\ex}(N,d):=\min_{\cP,\cA} L_{p,N}^{\ex}(\cP,\cA)$$ and the $N$-th minimal worst-case error is defined as $$e_q(N,d):=\min_{\cP,\cA} |e(Q_{\cP,\cA};F_{d,q})|,$$
where  $Q_{\cP,\cA}$ is a linear algorithm of the form \eqref{def:I:Q} and where in both cases the 
minimum is extended over all $N$-element point sets $\cP =\{\bsx_1,\ldots,\bsx_N\}$ from $[0,1)^d$ and real weights $\cA=\{c_1,\ldots,c_N\}$. 

Later we will restrict ourselves to non-negative weights $c_1,\ldots,c_N \ge 0$. We denote sets of such weights by $\cA^+$ and define
$${\rm disc}_p^{\ex,+}(N,d):=\min_{\cP,\cA^+} L_{p,N}^{\ex}(\cP,\cA^+)$$ and $$e_q^+(N,d):=\min_{\cP,\cA^+} |e(Q_{\cP,\cA^+};F_{d,q})|,$$
where in both cases the minimum is extended over all $N$-element point sets $\cP =\{\bsx_1,\ldots,\bsx_N\}$ from $[0,1)^d$ and non-negative weights $\cA^+$.

If we restrict to QMC weights $\cA^{\qmc}$, then $${\rm disc}_p^{\ex,\qmc}(N,d):=\min_{\cP} L_{p,N}^{\ex}(\cP)$$ and
$$e_q^{\qmc}(N,d):=\min_{\cP} |e(Q_{\cP,\cA^{\qmc}};F_{d,q})|$$ where in both cases the minimum is extended over all $N$-element point sets $\cP =\{\bsx_1,\bsx_2,\ldots,\bsx_N\}$ from $[0,1)^d$.

It is obvious, that $$e_q(N,d) \le e_q^+(N,d) \le e_q^{\qmc}(N,d)$$ and $${\rm disc}_p^{\ex}(N,d) \le {\rm disc}_p^{\ex,+}(N,d) \le {\rm disc}_p^{\ex,\qmc}(N,d).$$

From Theorem~\ref{thm:main} we obtain the following result:

\begin{cor}\label{cor1}
For H\"older conjugates $p,q$, for $d,N \in \N$ we have $$e_q^{\bullet}(N,d)={\rm disc}_p^{\ex,\bullet}(N,d) \qquad \mbox{for } \bullet \in \{{\rm blank},+,\qmc\}.$$
\end{cor}

\paragraph{Initial error and worst-case function.} The initial error is the worst-case error for the zero integration rule $Q\equiv 0$ which uses no point evaluation. It is denoted by $e(0;F_{d,q})$. Hence $$e(0;F_{d,q}) =\sup_{\|f\|_{F_{d,q},\bx} \le 1} |I(f)|.$$ 
Furthermore, we call a non-zero function $h \in F_{d,q}$ a worst-case function, if $$|I(h)|=  e(0;F_{d,q}).$$

\begin{prop}\label{pr:wc:fct}
Let $p \in [1,\infty)$ and let $q$ be the corresponding H\"older conjugate. We have $$e(0;F_{d,q})=L_{p,0}^{\ex}$$ and the function
$$h_d(\bsx)=\left(\frac{p+2}{p} \frac{1}{((p+1)(p+2))^{1/p}}\right)^d \prod_{j=1}^d \left(1-x_j^{p+1}-(1-x_j)^{p+1}\right)$$ is a worst-case function. We have $h_d \ge 0$, $\|h_d\|_{F_{d,q},\bx}=1$ and $$I(h_d) = L_{p,0}^{\ex}=\frac{1}{((p+1)(p+2))^{d/p}}.$$
\end{prop}

\begin{proof}
For the zero integration rule we have $\mu=-\lambda$ and we obtain from the first part of the proof of Theorem~\ref{thm:main} that $$\int_{[0,1]^d}f(\bsx)\rd \bsx = I(f)=L(f)=\int_{\D_d} c(\bsa,\bsb) \lambda([\bsa,\bsb]) \rd (\bsa,\bsb),$$ where $c \in L_q(\D_d)$ is any representer of $f$. Hence $$|I(f)| \le \|c\|_{L_q(\D_d)} \, \|\Delta_{\emptyset}\|_{L_p(\D_d)},$$ where obviously $\Delta_{\emptyset} (\bsa,\bsb)= -\lambda([\bsa,\bsb])$. Taking the infimum over all representers $c$ and noting that $\|\Delta_{\emptyset}\|_{L_p(\D_d)} =L_{p,0}^{\ex}$ we obtain $|I(f)| \le \|f\|_{F_{d,q},\bx} \, L_{p,0}^{\ex}$ and hence 
$$e(0; F_{d,q}) \le L_{p,0}^{\ex}.$$ 

Put 
\begin{align}\label{eq:cstar0}
c_0^{\ast}(\bsa,\bsb):= & \frac{|\Delta_{\emptyset}(\bsa,\bsb)|^{p-1}}{\|\Delta_{\emptyset}\|_{L_p(\D_d)}^{p-1}}=  \frac{(\lambda([\bsa,\bsb]))^{p-1}}{\|\Delta_{\emptyset}\|_{L_p(\D_d)}^{p-1}}
\end{align}
if $p>1$ and $c_0^{\ast}(\bsa,\bsb):= 1$ if $p=1$. Then $c_0^{\ast} \ge 0$ and $\|c_0^{\ast}\|_{L_q(\D_d)}=1$. Consider $h_d:=T_d c_0^{\ast}$. Then $\|h_d\|_{F_{d,q},\bx} \le \|c_0^{\ast}\|_{L_q(\D_d)}=1$ and
\begin{align*}
I(h_d)
= & \int_{\D_d} c_0^\ast(\bsa,\bsb)\,\lambda([\bsa,\bsb])\rd (\bsa,\bsb) \\
= & \frac{1}{\|\Delta_{\emptyset}\|_{L_p(\D_d)}^{p-1}}\int_{\D_d} (\lambda([\bsa,\bsb]))^p \rd (\bsa,\bsb)\\
= & \|\Delta_{\emptyset}\|_{L_p(\D_d)} = L_{p,0}^{\ex}\\
\ge & \|h_d\|_{F_{d,q},\bx} \,L_{p,0}^{\ex}.
\end{align*}
Therefore $e(0;F_{d,q}) \ge L_{p,0}^{\ex}$. Combining with the upper bound we obtain $$e(0;F_{d,q}) = L_{p,0}^{\ex} \qquad \mbox{and}\qquad \frac{I(h_d)}{\|h_d\|_{F_{d,q},\bx}}=e(0;F_{d,q}),$$ i.e., $h_d$ is a worst-case function. 
Note that from $\|h_d\|_{F_{d,q},\bx} \,L_{p,0}^{\ex} \ge I(h_d) \ge L_{p,0}^{\ex}$ it even follows that $\|h_d\|_{F_{d,q},\bx}=1$.

It remains to compute the explicit form of $h_d$. We have 
$$h_d(\bsx)=(T_d c_0^{\ast})(\bsx) =\int_{\D_d} c_0^{\ast}(\bsa,\bsb) {\bf 1}_{[\bsa,\bsb]}(\bsx) \rd (\bsa,\bsb).$$ With $c_0^{\ast}\ge 0$ also $h_d \ge 0$. From \eqref{eq:cstar0}, $\Delta_{\emptyset}(\bsa,\bsb)=-\lambda([\bsa,\bsb])$ and $\|\Delta_{\emptyset}\|_{L_p(\D_d)}=((p+1)(p+2))^{-d/p}$ we obtain 
\begin{align}\label{hd:expl}
c_0^{\ast}(\bsa,\bsb):=((p+1)(p+2))^{\frac{d (p-1)}{p}} \prod_{j=1}^d (b_j-a_j)^{p-1}.
\end{align}
Hence
\begin{align*}
h_d(\bsx) = & ((p+1)(p+2))^{\frac{d (p-1)}{p}}\int_{\D_d} \prod_{j=1}^d (b_j-a_j)^{p-1} {\bf 1}_{[\bsa,\bsb]}(\bsx) \rd (\bsa,\bsb)\\
= & ((p+1)(p+2))^{\frac{d (p-1)}{p}} \prod_{j=1}^d \int_{\D_1} (b-a)^{p-1} {\bf 1}_{[a,b]}(x_j) \rd(a,b).
\end{align*}
We have 
\begin{align*}
\int_{\D_1} (b-a)^{p-1} {\bf 1}_{[a,b]}(x_j) \rd(a,b) = & \int_0^{x_j} \int_{x_j}^1 (b-a)^{p-1} \rd b \rd a\\
= & \frac{1-x_j^{p+1}-(1-x_j)^{p+1}}{p(p+1)}.
\end{align*}
Thus $$h_d(\bsx)=\left(\frac{p+2}{p} \frac{1}{((p+1)(p+2))^{1/p}}\right)^d \prod_{j=1}^d \left(1-x_j^{p+1}-(1-x_j)^{p+1}\right)$$ as desired.
\end{proof}

\paragraph{Inverse discrepancy and information complexity.} For $d \in \N$ and $\varepsilon \in (0,1)$ the inverse of the $N$-th minimal extreme $L_p$ discrepancy is defined as $$N_p^{{\rm disc},\bullet}(\varepsilon,d):=\min\{N \in \N \ : \ {\rm disc}_p^{\bullet}(N,d) \le \varepsilon \ L_{p,0}^{\ex}\}$$ where $\bullet \in \{{\rm blank},+,\qmc\}$. The information complexity of the integration problem is defined as the minimal number of function evaluations necessary in order to reduce the initial error by a factor of $\varepsilon$, i.e., for $d \in \N$ and $\varepsilon \in (0,1)$,  
$$N^{{\rm int},\bullet}_q(\varepsilon,d):= \min\{N \in \N \ : \ e_q^{\bullet}(N,d) \le \varepsilon\ e_q(0,d)\},$$ where $\bullet \in \{{\rm blank},+,\qmc\}$.

Obviously, $$N_p^{{\rm disc}}(\varepsilon,d) \le N_p^{{\rm disc},+}(\varepsilon,d) \le N_p^{{\rm disc},\qmc}(\varepsilon,d)$$ and likewise for $N^{{\rm int},\bullet}_q(\varepsilon,d)$. From Corollary~\ref{cor1} we obtain:
\begin{cor}\label{cor2}
For H\"older conjugates $p,q$, for $d \in \N$ and $\varepsilon \in (0,1)$ we have $$N_p^{{\rm disc},\bullet}(\varepsilon,d) = N^{{\rm int},\bullet}_q(\varepsilon,d) \qquad \mbox{for  $\bullet \in \{{\rm blank},+,\qmc\}$.}$$
\end{cor}

\section{The curse of dimensionality}\label{sec:curse}

It is known from \cite{hnww} that the extreme $L_{\infty}$ discrepancy is polynomially tractable with the upper bound $N_\infty^{{\rm disc},\qmc}(\varepsilon,d) \le  C \, d \varepsilon^{-2}$. The proof is as for the star-discrepancy; see also Hinrichs~\cite{Hi04}. Gnewuch~\cite{Gne05,Gne08} proved explicit upper bounds. In particular, \cite[Theorem~2.2]{Gne08} states that $$N_\infty^{{\rm disc},\qmc}(\varepsilon,d) \le \left\lceil 2 
\varepsilon^{-2} \left(2 d \log\left(\frac{10 {\rm e}}{\varepsilon}\right)+\log 2\right)\right\rceil$$ 
for $d \in \mathbb{N}\setminus\{1\}$ and for all $\varepsilon \in (0,1)$.

On the other hand, the extreme $L_2$ discrepancy suffers from the curse of dimensionality, as shown in \cite[p.~93-94]{NW10}. This means that the inverse of the $N$-th minimal extreme $L_2$ discrepancy grows at least exponential in $d$. In more detail, 
$$N_2^{{\rm disc},+}(\varepsilon,d) \ge (1-\varepsilon^2) \left(\frac{9}{4}\right)^d.$$

The behavior for general $p \not\in \{2,\infty\}$ was unknown so far and is the topic of this section. As in the case of $p=2$ in \cite{NW10}, our method only allows a result for linear rules with non-negative weights. In any case, this includes the case of QMC rules. The case $p=1$ remains as an open problem.

\begin{thm}\label{thm:main2}
For every $p \in (1,\infty)$ there exists a number $C_p>1$ such that $$N_p^{{\rm disc},+}(\varepsilon,d) \ge C_p^d (1-2 \varepsilon) \qquad \mbox{for $d \in \mathbb{N}$ and $\varepsilon \in [0,\tfrac{1}{2})$.}$$ In other words, the extreme $L_p$ discrepancy for non-negative weights (and hence also for QMC weights) suffers from the curse of dimensionality.
\end{thm}

According to Corollary~\ref{cor2}, an equivalent formulation in terms of integration is the following.

\begin{thm}\label{thm:main3}
For every $q \in (1,\infty)$ we have $$N_q^{{\rm int},+}(\varepsilon,d) \ge C_p^d (1-2 \varepsilon) \qquad \mbox{for $d \in \mathbb{N}$ and $\varepsilon \in [0,\tfrac{1}{2})$,}$$ where $C_p>1$ is the number  from Theorem~\ref{thm:main2} with $p$ the H\"older conjugate of $q$.
In other words, integration in $F_{d,q}$ with non-negative linear rules suffers from the curse of dimensionality.
\end{thm}

The proof of Theorem~\ref{thm:main3}, and therefore also of Theorem~\ref{thm:main2}, is presented in the following section.

\begin{remark}\rm
It follows from \eqref{Bp8} in Appendix~\ref{sec:appB} that for $p \in \{2,3,\ldots,7\}$ we may take 
$$C_p =\frac{p}{p+2} \frac{2^p}{2^p-1} \left(\frac{(p+1)(p+2)}{4}\right)^{1/p}.$$ 
For any other $p \in (1,\infty)$ a suitable value can be determined numerically, see Figure~\ref{f1}. Note that $\lim_{p \rightarrow  1^+} C_p= \lim_{p \to \infty} C_p  = 1$.

\begin{figure}
\centering
\includegraphics[width=0.8 \textwidth]{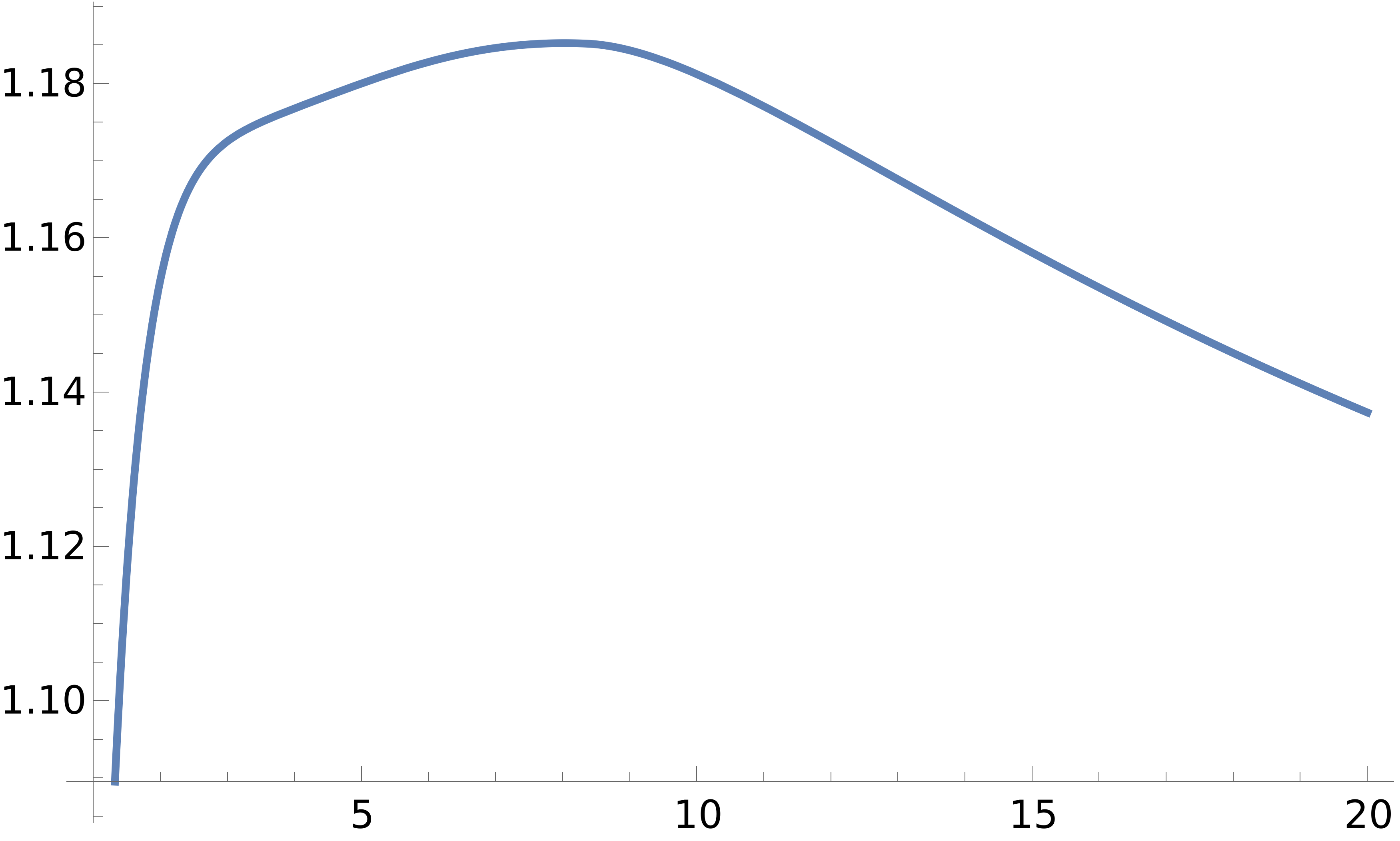}
\caption{The quantity $C_p$ for $p \in (1,20]$.}
\label{f1}
\end{figure}
\end{remark}

\begin{remark}\rm
Theorem~\ref{thm:main3} holds with  the same proof (but other constants $C_p$) also for 
the 
mixed Sobolev norm  
\[
\|f\|_{F_{d,q},\mx} := \| D f\|_{L_q([0,1]^d)},
\]
see Remark~\ref{remark3}. This generalizes the result from \cite[page~94, line~5]{NW10} from $p=2$ to general $p \in (1,\infty)$.
The stress in this paper is on the extreme discrepancy and it was a challenge to find the 
``right'' integration problem, i.e., the correct norm. The proof of the curse itself is even simpler for the standard mixed Sobolev norm. Here are just a few key facts. One uses the same arguments like in the proof in Section~\ref{sec:proof} below. The initial error is $e(0;F_{d,q}^{\mx})=(2 (p+1)^{1/p})^{-d}$ and a univariate worst-case function is $h_1(x)=1-|1-2 x|^p$ for $x \in [0,1]$. Then use the same splines as in \eqref{def:syexpl}. \iffalse functions $$s_y(x):=\left\{
\begin{array}{ll}
x \frac{h_1(y)}{\min(y,1-y)} & \mbox{for } 0 \le x \le \min(y,1-y),\\
h_1(y) & \mbox{for } \min(y,1-y)< x < \max(y,1-y),\\
(1-x) \frac{h_1(y)}{\min(y,1-y)} & \mbox{for } \max(y,1-y) \le x \le 1,
\end{array}
\right.$$ instead of \eqref{def:sy}.\fi
\end{remark}

\section{The proof of Theorem~\ref{thm:main3}}\label{sec:proof}

We use the method developed in \cite{NP25} (see also \cite{NP25b}) and always assume that $p,q \in (1, \infty)$.
This method is based on a spline interpolation for the worst-case function.

\paragraph{Spline interpolation for the worst-case function.}

We consider the univariate worst-case function 
\begin{equation}\label{def:h1}
h_1(x)=\frac{p+2}{p} \frac{1}{((p+1)(p+2))^{1/p}} \left(1-x^{p+1}-(1-x)^{p+1}\right).
\end{equation}
For $x \in [0,1]$ set $$E_x:=\{(a,b) \in \D_1 : a \le x \le b\}.$$ Then for every $c \in L_q(\D_1)$ we have $$(T_1c)(x) =\int_{\D_1} c(a,b) {\bf 1}_{[a,b]}(x) \rd (a,b)=\int_{E_x} c(a,b)\rd (a,b).$$ The area of $E_x$, denoted by $m(E_x)$, is $$m(E_x)= \int_0^x \int_x^1 \rd b \rd a = x (1-x).$$ In particular, $m(E_0)=m(E_1)=0$ and thus also $(T_1c)(0)=(T_1c)(1)=0$.

For $y \in [0,1]$ put $$c_y(a,b):=\frac{h_1(y)}{m(E_y)} {\bf 1}_{E_y}(a,b)\qquad \mbox{for $(a,b) \in \D_1$,}$$ and 
\begin{equation}\label{def:sy}
s_y(x):=(T_1 c_y)(x)=\int_{\D_1} c_y(a,b) {\bf 1}_{[a,b]}(x) \rd (a,b)=\int_{E_x} c_y(a,b)\rd (a,b),
\end{equation}
for which we can deduce also a more explicit presentation in the following way. Since $m(E_y)=y(1-y)$ we deduce $$s_y(x)=\frac{h_1(y)}{y (1-y)} m(E_x \cap E_y).$$ We have $$E_x \cap E_y=\{(a,b) \in \D_1 : a \le \min(x,y), b \ge \max(x,y)\}$$ and therefore $$m(E_x \cap E_y)=\min(x,y)(1-\max(x,y)).$$ Hence
\begin{eqnarray}\label{def:syexpl}
s_y(x) & = & \frac{h_1(y)}{y (1-y)}\min(x,y)(1-\max(x,y))\nonumber\\
& = & \left\{ 
\begin{array}{ll}
\frac{h_1(y)}{y} x & \mbox{if } x \in [0,y],\\[0.5em]
\frac{h_1(y)}{1-y}(1-x)  & \mbox{if } x \in [y,1].
\end{array}\right.
\end{eqnarray}

\begin{prop}\label{pr:s:form}
For the function $s_y$ defined by \eqref{def:sy} we have $s_y(x) \ge 0$ and $s_y(0)=s_y(1)=0$ and $$s_y(y)=h_1(y).$$ Furthermore, $$\|s_y\|_{F_{1,q},\bx}=\frac{h_1(y)}{(y(1-y))^{1/p}} \qquad \mbox{and}\qquad \int_0^1 s_y(x) \rd x = \frac{h_1(y)}{2}.$$
\end{prop}

\begin{proof}
The properties $s_y(x) \ge 0$ and $s_y(0)=s_y(1)=0$ and $s_y(y)=h_1(y)$ and the formula for the integral follow directly from \eqref{def:syexpl}.

We compute the $\|\cdot\|_{F_{1,q},\bx}$-norm of $s_y$. For any function $c \in L_q(\D_1)$ such that $\int_{E_y} c(a,b) \rd (a,b)=h_1(y)$ we have by H\"older 
\begin{equation}\label{Hoehd}
h_1(y)=\left|\int_{E_y} c(a,b) \rd (a,b)\right| \le \|c\|_{L_q(\D_1)} m(E_y)^{1/p}
\end{equation}
and hence 
\begin{equation*}
\|c\|_{L_q(\D_1)} \ge \frac{h_1(y)}{m(E_y)^{1/p}}.
\end{equation*}
On the other hand, 
\begin{align*}
\|c_y\|_{L_q(\D_1)}^q = & \int_{\D_1} \left(\frac{h_1(y)}{m(E_y)}\right)^q {\bf 1}_{E_y}(a,b) \rd(a,b)\\
= & \int_{E_y} \left(\frac{h_1(y)}{m(E_y)}\right)^q \rd(a,b) = \left(\frac{h_1(y)}{m(E_y)}\right)^q m(E_y).
\end{align*}
Hence $$\|c_y\|_{L_q(\D_1)}=\frac{h_1(y)}{m(E_y)^{1/p}}.$$ This implies that $$\|s_y\|_{F_{1,q},\bx}=\frac{h_1(y)}{m(E_y)^{1/p}}=\frac{h_1(y)}{(y(1-y))^{1/p}}.$$
\end{proof}

Using the formulas from Proposition~\ref{pr:s:form} one can bound the integral and norm of $s_y$. The proof uses elementary methods, but is very tedious and tricky, especially for the norm. The detailed proof can be found in Appendix~\ref{sec:appB}.

\begin{prop}\label{pr:big1}
For the function $s_y$ defined by \eqref{def:sy} we have $$A_p:=\max_{y \in [0,1]} \frac{\int_0^1 s_y(x)\rd x}{\int_0^1 h_1(x)\rd x}=\frac{p+2}{2 p} \left(1-\frac{1}{2^p}\right) < 1 \quad \mbox{and}\quad B_p:=\max_{y \in [0,1]} \frac{\|s_y\|_{F_{1,q},\bx}}{\|h_1\|_{F_{1,q},\bx}} < 1.$$
\end{prop}

\paragraph{The proof of Theorem~\ref{thm:main3}.} After all our preparations, we can now apply the method from \cite{NP25} to prove our theorem.

\begin{proof}[Proof of Theorem~\ref{thm:main3}]
Consider a linear algorithm $Q_{\cP,\cA}$ of the form \eqref{def:I:Q} based in nodes $\cP=\{\bsx_1,\ldots,\bsx_N\}$ in $[0,1]^d$ and with non-negative weights $\cA^+=\{c_1,\ldots,c_N\}$.  For $k \in \{1,\ldots,N\}$ and $j \in \{1,\ldots,d\}$ let $x_{k,j}$ be the $j$-th coordinate of the point $\bsx_k$. For $k \in \{1,\ldots,N\}$ we define functions 
$$
P_k(\bsx) := s_{x_{k,1}}(x_1) s_{x_{k,2}}(x_2) \cdots s_{x_{k,d}}(x_d),\quad \mbox{for $\bsx=(x_1,\ldots,x_d)\in [0,1]^d$.}  
$$

Consider the two functions $h_d$ (the worst-case function from Proposition~\ref{pr:wc:fct}) and $f^*:= \sum_{i=1}^N P_i$.
Since $Q_{\cP,\cA^+}$ uses only non-negative weights we have 
\begin{align*}
Q_{\cP,\cA^+}(f^*) = & \sum_{k=1}^N c_k \sum_{j=1}^N P_j(\bsx_k) \ge \sum_{k=1}^N c_k P_k(\bsx_k)\\
= & \sum_{k=1}^N c_k s_{x_{k,1}}(x_{k,1}) s_{x_{k,2}}(x_{k,2}) \cdots s_{x_{k,d}}(x_{k,d}) \\
= & \sum_{k=1}^N c_k h_d(\bsx_k) = Q_{\cP,\cA^+}(h_d).
\end{align*}

Now for real $y$ we use the notation $(y)_+:= \max (y, 0)$.  Then we have the error estimate 
\begin{equation}\label{errest1}
e(Q_{\cP,\cA^+};F_{d,q})  \ge \frac{ \left(I(h_d) - I(f^*)\right)_+}{2 \max ( 1, \| f^* \|_{F_{d,q},\bx})},
\end{equation}
which is trivially true if $I(h_d) \le  I(f^*)$ and which is easily shown if $I(h_d) >  I(f^*)$, because then 
\begin{align*}
\left(I(h_d) - I(f^*)\right)_+ \le & I(h_d) - Q_{\cP,\cA^+}(h_d)+Q_{\cP,\cA^+}(f^*)-I(f^*)\\
\le & \|h_d\|_{F_{d,q},\bx} \, e(Q_{\cP,\cA^+};F_{d,q})+\|f^*\|_{F_{d,q},\bx} \, e(Q_{\cP,\cA^+};F_{d,q})\\
\le & 2 \max(1,\|f^*\|_{F_{d,q},\bx}) \, e(Q_{\cP,\cA^+};F_{d,q}),
\end{align*}
where we used that $\|h_d\|_{F_{d,q},\bx}=1$. This implies \eqref{errest1}.

From the triangle inequality and from the definition of $A_p$ and $B_p$ in Proposition~\ref{pr:big1} we obtain 
$$
\| f^* \|_{F_{d,q},\bx} \le N B_p^d \qquad \mbox{and} \qquad I(f^*) \le N  A_p^d \, I(h_d).
$$ 
Inserting into \eqref{errest1} and taking into account that $I(h_d)= e(0;F_{d,q})$ yields 
$$
e(Q_{\cP,\cA^+};F_{d,q}) \ge \frac{ e(0;F_{d,q})  \left(1- N A_p^d\right)_+} {2 \max (1, 
N B_p^d) }.
$$
This yields  
\begin{equation*}
e_q^+(N,d)\ge \frac{ e(0;F_{d,q})  \left(1- N A_p^d\right)_+} {2 \max (1, 
N B_p^d) }.
\end{equation*}

Now let $\varepsilon \in (0,1/2)$ and assume that $e_q^+(N,d) \le \varepsilon \, e_q(0,d)$. This implies that
$$2 \varepsilon  \max \left(1, 
N B_p^d\right)  \ge  \left(1 - N A_p^d\right)_+.$$

Proposition~\ref{pr:big1} yields
\begin{equation*}%\label{def:C1new}
C_p:=\min\left(\frac{1}{A_p},\frac{1}{B_p}\right) >1.
\end{equation*}

If $N \le C_p^d$, then we obtain
\begin{align*}
1- N A_p^d  = & \left(1 - N A_p^d\right)_+ \le 2 \varepsilon \max \left(1, N B_p^d\right) =  2 \varepsilon. 
\end{align*}
Hence $$N \ge \frac{1}{A_p^d} (1-2 \varepsilon) \ge C_p^d (1-2 \varepsilon).$$
If $N \ge C_p^d$, then we trivially have $N \ge C_p^d (1-2 \varepsilon)$.
 
This yields $$N_q^{{\rm int},+}(\varepsilon,d)\ge C_p^d (1-2 \varepsilon),$$ and we are done.
\end{proof}

\begin{appendix}
\section{Appendix}\label{sec:appA}

\begin{proof}[Proof of the norm-equivalence \eqref{norm:equiv}]
For $d=1$ we have $f(x)=\int_0^x \int_x^1 c(a,b) \rd b \rd a$ and hence by Leibniz integration rule $$f'(x)=\int_x^1 c(x,b) \rd b - \int_0^x c(a,x) \rd a.$$ In order to extend this to the multivariate case we need some notation. Let $[d]:=\{1,\ldots,d\}$. For a subset $\uu \subseteq [d]$ we write $|\uu|$ for the number of elements of $\uu$, $\uu^c:=[d]\setminus \uu$, and, for two vectors $\bsx,\bsy \in [0,1]^d$, $(\bsx_\uu,\bsy_{\uu^c}):=(z_1,\ldots,z_d)$ where 
$$z_j=\left\{ 
\begin{array}{ll}
x_j & \mbox{if } j \in \uu,\\
y_j & \mbox{if } j \in \uu^c.
\end{array}\right.$$
Furthermore, let $B_{\uu}(\bsx):= B_1 \times \ldots \times B_d$, where 
$$B_j=\left\{ 
\begin{array}{ll}
[0,x_j] & \mbox{if } j \in \uu,\\
{[}x_j,1]  & \mbox{if } j \in \uu^c.
\end{array}\right.$$

Then for $f=T_d c$ we obtain  
\begin{equation}\label{sum:Df}
D f(\bsx)=\sum_{\uu \subseteq [d]} (-1)^{|\uu|} \int_{B_{\uu}(\bsx)} c((\bsa_{\uu},\bsx_{\uu^c}),(\bsx_{\uu},\bsb_{\uu^c})) \rd (\bsa_{\uu},\bsb_{\uu^c}),
\end{equation}
where $D$ again denotes the operator $D f=\partial_1\cdots\partial_d f$. The $2^d$ terms of this sum are fiber integrals of $c$ over the $2^d$ faces of the $2d$-dimensional set $\{(\bsa,\bsb) : \bsx \in [\bsa,\bsb]\}$. Now we show that each of these $2^d$ fiber maps is a bounded operator $T_{\uu}: L_q(\D_d) \rightarrow L_q([0,1]^d)$ $$(T_{\uu} c)(\bsx)=\int_{B_{\uu}(\bsx)} c((\bsa_{\uu},\bsx_{\uu^c}),(\bsx_{\uu},\bsb_{\uu^c})) \rd (\bsa_{\uu},\bsb_{\uu^c}),\quad (\mbox{for }\uu \subseteq [d])$$ with operator norm at most $1$. Indeed, using H\"older's inequality we have
$$|(T_{\uu}c)(\bsx)| \le \left(\int_{B_{\uu}(\bsx)} | c((\bsa_{\uu},\bsx_{\uu^c}),(\bsx_{\uu},\bsb_{\uu^c}))|^q \rd (\bsa_{\uu},\bsb_{\uu^c})\right)^{1/q}$$ and hence
$$\|T_{\uu} c\|_{L_q([0,1]^d)}^q \le \int_{[0,1]^d} \int_{B_{\uu}(\bsx)} | c((\bsa_{\uu},\bsx_{\uu^c}),(\bsx_{\uu},\bsb_{\uu^c}))|^q \rd (\bsa_{\uu},\bsb_{\uu^c}) =\|c\|_{L_q(\D_d)}^q.$$
Therefore we obtain  $\|T_{\uu} c\|_{L_q([0,1]^d)} \le \|c\|_{L_q(\D_d)}$.

Using \eqref{sum:Df} and the triangle inequality we obtain 
\begin{equation*}\label{bd:norm:u}
\|f\|_{F_{d,q},\mx}=\|D f\|_{L_q([0,1]^d)} \le 2^d \|c\|_{L_q(\D_d)}.
\end{equation*}
This proves the left inequality in \eqref{norm:equiv}.

In order to prove the right inequality in \eqref{norm:equiv} let, for given $f \in F_{d,q}$, $g:=Df \in L_q([0,1]^d)$ and 
\begin{equation}\label{def:cg}
c_g(\bsa,\bsb):=\sum_{\uu \subseteq [d]}(-1)^{d-|\uu|} g(\bsa_{\uu},\bsb_{\uu^c}).
\end{equation}

It is easy to see that $${\bf 1}_{\{t \le x\}}-x = \int_{\D_1} (\delta(t-a)-\delta(t-b)) {\bf 1}_{[a,b]}(x) \rd (a,b),$$ where $\delta$ denotes the Dirac distribution. Then we have 
\begin{align*}
\prod_{j=1}^d \left({\bf 1}_{\{t_j \le x_j\}}-x_j\right) = & \int_{\D_d} {\bf 1}_{[\bsa,\bsb]}(\bsx) \prod_{j=1}^d (\delta(t_j-a_j)-\delta(t_j-b_j)) \rd (\bsa,\bsb)\\
= & \int_{\D_d} {\bf 1}_{[\bsa,\bsb]}(\bsx) \sum_{\uu \subseteq [d]} (-1)^{d-|\uu|} \prod_{j \in \uu} \delta(t_j-a_j) \prod_{j \in \uu^c} \delta(t_j-b_j) \rd (\bsa,\bsb).
\end{align*}
Furthermore, because of the boundary conditions in the univariate ($d=1$) case it is easy to see that $f(x)=\int_0^1 ({\bf 1}_{\{t \le x\}}-x) g(t)\rd t$. Hence
\begin{align*}
f(\bsx) = & \int_{[0,1]^d} \prod_{j=1}^d \left({\bf 1}_{\{t_j \le x_j\}}-x_j\right) g(\bst) \rd \bst\\
= & \int_{\D_d} {\bf 1}_{[\bsa,\bsb]}(\bsx) \sum_{\uu \subseteq [d]} (-1)^{d-|\uu|} \left( \int_{[0,1]^d} \prod_{j \in \uu} (\delta(t_j-a_j) \prod_{j \in \uu^c} \delta(t_j-b_j) g(\bst)\rd \bst\right) \rd (\bsa,\bsb)\\
= & \int_{\D_d} {\bf 1}_{[\bsa,\bsb]}(\bsx) \sum_{\uu \subseteq [d]} (-1)^{d-|\uu|} g(\bsa_{\uu},\bsb_{\uu^c}) \rd (\bsa,\bsb)\\
= & \int_{\D_d} {\bf 1}_{[\bsa,\bsb]}(\bsx) c_g(\bsa,\bsb) \rd (\bsa,\bsb).
\end{align*}

We study the $L_q(\D_d)$-norm of $c_g$. First of all, for any $\uu \subseteq [d]$ we have 
\begin{align*}
\int_{\D_d} | g(\bsa_{\uu},\bsb_{\uu^c})|^q \rd (\bsa,\bsb) =& \int_{[0,1]^d} \left(\prod_{j \in \uu}(1-a_j) \prod_{j \in \uu^c} b_j\right)  | g(\bsa_{\uu},\bsb_{\uu^c})|^q \prod_{j \in \uu} \rd a_j \prod_{j \in \uu^c} \rd b_j\\
\le &  \int_{[0,1]^d} | g(\bst)|^q \rd \bst\\
= & \|g\|_{L_q([0,1]^d)}^q = \|D f\|_{L_q([0,1]^d)}^q.
\end{align*}
Using \eqref{def:cg} and the triangle inequality we thus obtain 
\begin{equation*}\label{bd:norm:o}
\|f\|_{F_{d,q},\bx} \le \|c_g\|_{L_q(\D_d)} \le 2^d \|D f\|_{L_q([0,1]^d)} = 2^d \|f\|_{F_{d,q},\mx}.
\end{equation*}
This proves the right inequality in \eqref{norm:equiv} and finishes the proof for general $q$.

Now we deal with the case $q=2$. 
The squared worst-case error in a  reproducing kernel Hilbert space  $F_{d,2}$ can be written as 
\[
e(Q_{\cP,\cA};F_{d,2})^2 = \Vert h_d \Vert^2 - 2 \sum _{k=1}^N c_k h_d (  \bsx_k) + \sum_{j,k=1}^N c_j c_k K_d ( \bsx_j, \bsx_k ) ,
\]
see, e.g., \cite[Eq.~(9.31)]{NW10}. Hence the error is given by the kernel and this formula also determines  the kernel 
if all worst case errors are known. 
Now it is known that 
\[
K_d^\mx (\bsx, \bsy) = \prod_{k=1}^d ( \min (x_k, y_k) - x_k y_k) 
\]
is the kernel of $F_{d,2}$ with the mix norm and yields the correct error for the extreme $L_2$ discrepancy, see \cite[Eq.~(9.41)]{NW10}. 
Hence the dual integration problem is already known for $q=2$. 
Together with Theorem 2 we obtain that $K_d^{\mx}$ is also the kernel of $F_{d,2}$ with the box norm 
and the norms coincide. Both norms or kernels  
describe the same integration 
problem. 
We add that a more direct calculation of the kernel $K_d^{\bx} $ is possible, 
based on the fact that the infimum in the definition of 
$
\|f\|_{F_{d,q},\bx} 
$   
can be easily described in a Hilbert space via orthogonality. 
\end{proof}

\section{Appendix}\label{sec:appB}

\begin{proof}[Proof of Proposition~\ref{pr:big1}]
Fix $p \in (1,\infty)$. First we deal with $A_p$. We have 
$$\frac{\int_0^1 s_y(x)\rd x}{\int_0^1 h_1(x)\rd x}=\frac{h_1(y)/2}{((p+1)(p+2))^{-1/p}}$$ and hence
$$\max_{y \in [0,1]} \frac{\int_0^1 s_y(x)\rd x}{\int_0^1 h_1(x)\rd x}= \frac{((p+1)(p+2))^{1/p}}{2} \max_{y \in [0,1]} h_1(y).$$ 
From \eqref{def:h1} it is easily seen that $$\max_{y \in [0,1]} h_1(y)=h_1(\tfrac{1}{2})= \frac{p+2}{p} \frac{1}{((p+1)(p+2))^{1/p}} \left(1-\frac{1}{2^p}\right)$$ and hence $$A_p=\max_{y \in [0,1]} \frac{\int_0^1 s_y(x)\rd x}{\int_0^1 h_1(x)\rd x}=\frac{p+2}{2 p} \left(1-\frac{1}{2^p}\right) < 1.$$

Now we show that $B_p<1$. We have $\|h_1\|_{F_{1,q},\bx}=1$ and $$\|s_y\|_{F_{1,q},\bx}=\frac{h_1(y)}{(y(1-y))^{1/p}}.$$ Thus, $$B_p=\max_{y \in [0,1]} \|s_y\|_{F_{1,q},\bx}=\max_{y \in [0,1]}\frac{h_1(y)}{(y(1-y))^{1/p}}.$$ First we show that for all $y \in (0,1)$ we have $\|s_y\|_{F_{1,q},\bx} < 1$. Using the representer $c_0^{\ast}(a,b)=((p+1)(p+2))^{(p-1)/p}(b-a)^{p-1}$ from \eqref{hd:expl} of the univariate worst-case function $h_1$ we obtain from \eqref{Hoehd} that $$h_1(y) \le \left(\int_{E_y} |c_0^{\ast}(a,b)|^q \rd(a,b)\right)^{1/q} m(E_y)^{1/p}.$$ Since $c_0^{\ast}(a,b)>0$ for almost all $(a,b) \in \D_1$ and $E_y$ is a proper measurable subset of $\D_1$ with positive measure we obtain $$\int_{E_y} |c_0^{\ast}(a,b)|^q \rd(a,b) < \int_{\D_1} |c_0^{\ast}(a,b)|^q \rd(a,b)=\|c_0^{\ast}\|_{L_q(\D_1)}^q=1.$$ Hence, $h_1(y) < m(E_y)^{1/p}=(y(1-y))^{1/p}$. This yields $$\|s_y\|_{F_{1,q},\bx}=\frac{h_1(y)}{(y(1-y))^{1/p}} <1$$ pointwise for any $y \in (0,1)$, as claimed. 

For $y \in (0,1)$ put $F_p(y):=\frac{h_1(y)}{(y(1-y))^{1/p}}$. Thus $\|s_y\|_{F_{1,q},\bx}=F_p(y)$ and $F_p(y)$ is continuous on $(0,1)$.

As $y \downarrow 0$, $$1-y^{p+1}-(1-y)^{p+1}=(p+1)y+O(y^2),$$ so $F_p(y)=O(y^{1-1/p}) \rightarrow 0$, because $p>1$. By symmetry the same holds as $y \uparrow 1$. Thus $F_p$ extends continuously to $[0,1]$ by setting $$F_p(0)=F_p(1)=0.$$ Since $[0,1]$ is compact, $F_p$ attains its maximum. The maximum cannot be attained at $0$ or $1$, because there the value is zero. If it is attained at some point $y \in (0, 1)$, then the pointwise strict inequality just proved shows that the maximum value is strictly smaller than one. Hence $$B_p = \max_{y \in [0,1]} F_p(y) < 1.$$
\end{proof}

\begin{remark}\rm
Although $$F_p(y)=\frac{1-y^{p+1}-(1-y)^{p+1}}{y(1-y))^{1/p}}$$ has a seemingly simple representation, we encounter major problems with the maximization problem $F_p \rightarrow \max$ over $[0,1]$. By symmetry, any extremum of $F_p(y)$ must occur at $y=\tfrac{1}{2}$ or come in symmetric pairs $y$ and $1-y$. A first problem seems to be that the behavior of the function $F_p$ differs fundamentally for small values of $p$ from that for larger $p$. The change happens for $p$ somewhere in the interval $(8,9)$ (see Figure~\ref{f2}). For $p \in (1,8]$ the graphics suggest that the global maximum is attained in the single stationary point $y=\tfrac{1}{2}$ which would lead to $$F_p(y) \le F_p(\tfrac{1}{2})=\frac{p+2}{p} \frac{1}{((p+1)(p+2))^{1/p}} 4^{1/p} \left(1-\frac{1}{2^p}\right) < 1.$$ For $p \in \{2,3,\ldots,7\}$ this can be proven directly, but, unfortunately, we failed for general $p \in (1,8]$. For $p > 8$ the problem becomes even more delicate. As the graphs suggest, we then have more than one stationary point in $(0,1)$, and these cannot possibly be determined explicitly. However, in this range of $p$ we are able to prove sufficiently good upper bounds. In summary, we are able to show the following explicit effective bounds 
\begin{equation}\label{Bp8}
B_p=\max_{y \in [0,1]} F_p(y) \le \frac{(p+2)^{1-1/p}}{p}\times \left\{
\begin{array}{ll}
\left(\frac{4}{p+1}\right)^{1/p} \left(1-\frac{1}{2^p}\right) & \mbox{if } p \in \{2,3,\ldots,7\},\\[0.5em]
\left(\frac{p-1}{p}\right)^{1-1/p} 4^{1/p} & \mbox{if } p \in [8,11),\\[0.5em]
2^{1/p} & \mbox{if } p \in [11,\infty),
\end{array}\right.
\end{equation} 
(even with equality in the case $p \in \{2,3,\ldots,7\}$) and in any case the bound on the right hand side is strictly less than 1.
\end{remark}

\begin{figure}
\centering
\includegraphics[width=1.0 \textwidth]{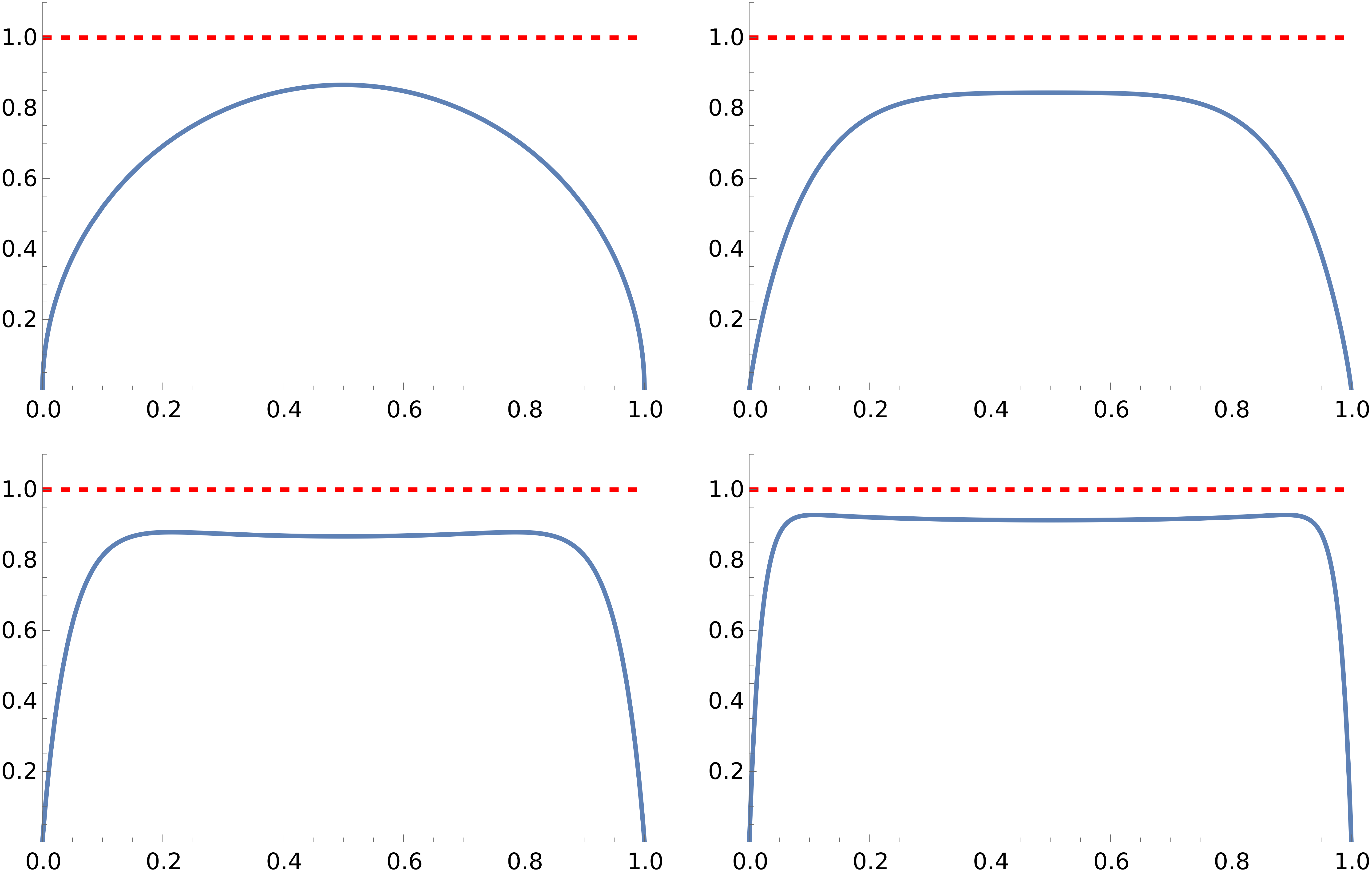}
\caption{The function $F_p(y)$, $y \in [0,1]$ for $p \in \{2,8\}$ (first row -- the maximum is attaind in $y=\tfrac{1}{2}$) and $p \in \{20,50\}$ (second row -- the maximum is {\it not} attained in $y=\frac{1}{2}$, rather in $y=\tfrac{1}{2}$ a local minimum is attained).}
\label{f2}
\end{figure}

\begin{proof}[Proof of Eq.~\eqref{Bp8}]
We consider the three cases separately.
\begin{description}
\item[$p \in \{2,3,\ldots,7\}$:] Write $F_p(y)=\frac{p+2}{p} \frac{1}{((p+1)(p+2))^{1/p}} S_p(y)$, where $S_p(y):=\frac{1-y^{p+1}-(1-y)^{p+1}}{(y(1-y))^{1/p}}$. We define  $L_p(y):=\log S_p(y)$. If we have
\begin{equation}\label{defLprime}
L_p'(y) =\frac{(p+1)((1-y)^p-y^p)}{1-y^{p+1}-(1-y)^{p+1}}-\frac{1}{p} \frac{1-2 y}{y(1-y)}\ge 0
\end{equation}
for $y \in (0,1/2]$, then, by symmetry of $S_p$, the maximum of $S_p$ on $[0,1]$ is attained at $y=1/2$. Hence, as long as \eqref{defLprime} is satisfied, we have 
$$
   S_p(y)\le S_p(1/2)=4^{1/p}\left(1-\frac1{2^p}\right)
$$
and thus $$F_p(y) \le \frac{p+2}{p} \frac{1}{((p+1)(p+2))^{1/p}} 4^{1/p} \left(1-\frac{1}{2^p}\right) < 1.$$

Obviously, \eqref{defLprime} is satisfied if $$E_p(y):=p (p+1)((1-y)^p-y^p)-\frac{1-2 y}{y(1-y)}(1-y^{p+1}-(1-y)^{p+1}) \ge 0.$$
For $p\in \{2,3,4,5,6,7\}$ we have 
\begin{align*}
E_2(y) = & 3-6y,\\
E_3(y) = & 8-26 y+30y^2-20 y^3,\\
E_4(y) = & 15-65 y+105 y^2-70y^3,\\
E_5(y) = & 24-129 y+271 y^2-274 y^3+140 y^4-56 y^5,\\
E_6(y) = & -7(-5+32 y-83 y^2+112 y^3-85 y^4+34 y^5),\\
E_7(y) = & -2(-24+178y-550 y^2+927 y^3-935 y^4+563 y^5-189 y^6+54 y^7),
\end{align*}
and any of these is non-negative over $(0,1/2]$.

\item[$p \in [11,\infty)$:] Let $y \in (0,1/2]$ and set $a:=(p+1) y$. Then $a \in (0,(p+1)/2]$. Since $y^{p+1}\ge 0$ we have $1-y^{p+1}-(1-y)^{p+1} \le 1-(1-y)^{p+1}$. Since $y \in (0,1/2] \subseteq (0,1)$ we have $\log(1-y)\ge -y/(1-y) \ge -2 y$ it follows that $$(1-y)^{p+1} \ge {\rm e}^{-2(p+1) y}={\rm e}^{-2a}.$$ Thus, $$1-(1-y)^{p+1} \le 1-{\rm e}^{-2 a}.$$ On the other hand, for $y \le 1/2$ we have $1-y \ge 1/2$ and hence $y(1-y) \ge y/2=a/(2(p+1))$. Therefore $$((p+1)(p+2) y (1-y))^{1/p} \ge \left((p+1)(p+2)\frac{a}{2(p+1)} \right)^{1/p}= \left(\frac{(p+2) a}{2} \right)^{1/p}.$$ Together we obtain
\begin{equation}\label{def:Gp}
F_p(y) \le \frac{p+2}{p} \left(1-{\rm e}^{-2 a}\right)\left(\frac{2}{(p+2)a}\right)^{1/p}=:G_p(a).
\end{equation}
Thus it suffices to prove that $G_p(a) < 1$ for all $a \in (0,(p+1)/2]$. 

Consider 
\begin{equation}\label{logGp}
\log G_p(a) =\log\frac{p+2}{p} + \log\left(1-{\rm e}^{-2 a}\right)+\frac{\log 2-\log(p+2)-\log a}{p}.
\end{equation}

Now we show that 
\begin{equation}\label{tem1le0}
\log\left(1-{\rm e}^{-2 a}\right)-\frac{\log a}{p} < 0.
\end{equation}
For $a \ge 1$ the inequality \eqref{tem1le0} trivially holds true. 

For $2^{-p/(p-1)} \le a \le 1$ we have $\tfrac{1}{p}\log a \ge \tfrac{1}{p}\tfrac{p}{p-1} \log \tfrac{1}{2} = \tfrac{1}{p-1}\log \tfrac{1}{2}$ and hence $$-\frac{\log a}{p} \le \frac{\log 2}{p-1} \le \frac{\log 2}{7},$$ since $p \ge 8$. Furthermore, from $a \le 1$ we obtain $$\log\left(1-{\rm e}^{-2 a}\right) \le \log\left(1-{\rm e}^{-2}\right).$$ Together this gives
\begin{align*}
\log\left(1-{\rm e}^{-2 a}\right)-\frac{\log a}{p} \le \log\left(1-{\rm e}^{-2}\right)+\frac{\log 2}{7} =-0.0463\ldots < 0,
\end{align*}
and hence \eqref{tem1le0} is shown also in this case.

For $a < 2^{-p/(p-1)}$ we obtain $a^{(p-1)/p} < \tfrac{1}{2}$ and hence $2 a < a^{1/p}$ and further $-a^{1/p} < -2a$. This yields $\log\left(1-a^{1/p}\right) \le - a^{1/p} < -2a$ and hence $1-a^{1/p}< {\rm e}^{-2 a}$. Hence $1-{\rm e}^{-2 a} < a^{1/p}$ and thus $$\log\left(1-{\rm e}^{-2 a}\right) < \frac{\log a}{p},$$ which proves \eqref{tem1le0} also in the final case.

Combining \eqref{logGp} and \eqref{tem1le0} we obtain
$$\log G_p(a) \le \log\frac{p+2}{p} +\frac{\log 2-\log(p+2)}{p} < 0$$ for all $p \in [11,\infty)$. Hence, for $p \in [11,\infty)$ and any $a \in (0,(p+1)/2]$ we have $G_p(a) < 1$.

\item[$p \in [8,11)$:] 
We continue studying $G_p(a)$ from \eqref{def:Gp}. It is easily seen that $G_p'(a)=0$ if and only if $a=a_{\ast}$ where $a_{\ast}$ satisfies $1-{\rm e}^{2 a_{\ast}}+2 a_{\ast} p=0$. This $a_{\ast}$ is the unique stationary point. We obtain $$1-{\rm e}^{-2 a_{\ast}} =\frac{2 a_{\ast} p}{1+2 a_{\ast} p}.$$ Let $$\widetilde{G}_p(a):=\frac{p+2}{p} \frac{2 a p}{1+2 a p}\left(\frac{2}{(p+2)a}\right)^{1/p}.$$ Then we have $$G_p(a) \le G_p(a_{\ast})=\widetilde{G}_p(a_{\ast}).$$

Now we maximize $\widetilde{G}_p(a)$. It is easily seen that $\widetilde{G}'_p(a)=0$ if and only if $a=\frac{p-1}{2 p}$. Again, this is the unique stationary point and we have $$G_p(a) \le \widetilde{G}_p(a_{\ast}) \le \widetilde{G}_p\left(\tfrac{p-1}{2 p}\right) = \left(\frac{(p-1)(p+2)}{p}\right)^{1-1/p} \frac{4^{1/p}}{p} < 1$$ for $p \in [8,11)$. 
\end{description}
\end{proof}
\end{appendix}

%\vspace{0.5cm}
\noindent{\bf Author's Address:}

\noindent Erich Novak, Mathematisches Institut, FSU Jena, Inselplatz 5, 07743 Jena, Germany. Email: erich.novak@uni-jena.de\\

\noindent Friedrich Pillichshammer, Institut f\"{u}r Finanzmathematik und Angewandte Zahlentheorie, JKU Linz, Altenbergerstra{\ss}e 69, A-4040 Linz, Austria. Email: friedrich.pillichshammer@jku.at

\end{document}